\documentclass[12pt,reqno]{amsart}
\usepackage{graphicx}
\usepackage{psfrag}
\usepackage{pxfonts}
\usepackage{mathrsfs}
\usepackage{color}
\usepackage{amsmath,amssymb}

\numberwithin{equation}{section}

\newcommand{\x} {\tilde{x}}

\theoremstyle{plain}
\newtheorem{maintheorem}{Theorem}

\newtheorem{theorem}{Theorem}[section]

\newtheorem{proposition}[theorem]{Proposition}
\newtheorem{lemma}[theorem]{Lemma}
\newtheorem{definition}[theorem]{Definition}

\newtheorem{question}{Question}
\theoremstyle{remark}
\newtheorem{remark}[theorem]{Remark}

\begin{document}

\thanks{First version: Arxiv May/2020. https://arxiv.org/pdf/2006.00407.pdf}

\author[F. Micena]{Fernando Micena}
\address{Instituto de Matem\'{a}tica e Computa\c{c}\~{a}o,
  IMC-UNIFEI, Itajub\'{a}-MG, Brazil.}
\email{fpmicena82@unifei.edu.br}


\renewcommand{\subjclassname}{\textup{2000} Mathematics Subject Classification}

\date{\today}

\setcounter{tocdepth}{2}

\title{Rigidity for Some Cases of Anosov Endomorphisms of Torus}
\maketitle
\begin{abstract}
 We obtain smooth conjugacy between non-necessarily special Anosov endomorphisms in the conservative case.  Among other results, we prove that a strongly special $C^{\infty}-$Anosov endomorphism of $\mathbb{T}^2$ and its linearization are smoothly conjugated since they have the same periodic data. Assuming that for a strongly special $C^{\infty}-$Anosov endomorphism of $\mathbb{T}^2$ every point is regular (in Oseledec's Theorem sense), then we obtain again smooth conjugacy with its linearization.
 We also obtain some results on local rigidity of linear Anosov endomorphisms of $d-$torus, where $d \geq 3,$ under periodic data assumption. The study of differential equations defined on invariant leaves plays an important role in rigidity problems such as those treated here.

\end{abstract}

\section{Introduction}\label{section.preliminaries}

In the $1970s,$ the works \cite{PRZ} and \cite{MP}  generalized the notion of Anosov diffeomorphism for non-invertible maps, introducing the notion of Anosov endomorphism. We consider $M$ a $C^{\infty}-$closed manifold.

\begin{definition}\cite{PRZ} \label{defprz} Let $f: M \rightarrow M$ be a  $C^1$ local diffeomorphism. We say that $f$ is an Anosov endomorphism if there are constants $C> 0$ and $\lambda > 1,$ such that, for every $(x_n)_{n \in \mathbb{Z}}$ an $f-$orbit there is a splitting

$$T_{x_i} M = E^s_{x_i} \oplus E^u_{x_i}, \forall i \in \mathbb{Z},$$

which is preserved by $Df$ and for all $n > 0 $ we have

$$||Df^n(x_i) \cdot v|| \geq C^{-1} \lambda^n ||v||, \;\mbox{for every}\; v \in E^u_{x_i} \;\mbox{and for any} \; i \in \mathbb{Z},$$
$$||Df^n(x_i) \cdot v|| \leq C\lambda^{-n} ||v||, \;\mbox{for every}\; v \in E^s_{x_i} \;\mbox{and for any} \; i \in \mathbb{Z}.$$

\end{definition}

We denote by $M^f$ the space of all $f-$orbits $\x= (x_n)_{n \in \mathbb{Z}},$ endowed with me metric $$\bar{d}(\tilde{x}, \tilde{y}) =  \sum_{i \in \mathbb{Z}} \frac{d(x_i, y_i)}{2^{|i|}},$$ where $d$ denotes the Riemannian metric on $M$ and $\x= (x_n)_{n \in \mathbb{Z}}, \tilde{y}= (y_n)_{n \in \mathbb{Z}},$ two $f-$orbits. We denote by $p: M^f \rightarrow M,$ the natural projection $$p((x_n)_{n \in \mathbb{Z}}) = x_0.$$

The space $(M^f, \bar{d})$ is compact, moreover $f$ induces a continuous map $\tilde{f}: M^f \rightarrow M^f,$ given by the shift
$$\tilde{f}((x_n)_{n \in \mathbb{Z}}) = (x_{n+1})_{n \in \mathbb{Z}}. $$

Anosov endomorphisms can be defined in an equivalent way (\cite{MP}).

\begin{definition}\cite{MP} \label{defmp} A $C^1$ local diffeomorphism $f: M \rightarrow M$ is said an Anosov endomorphism if $Df$ contracts uniformly a $Df-$invariant and continuous sub-bundle $E^s \subset TM$ into itself and the action of $Df$ on the quotient $TM/E^s$ is uniformly expanding.
\end{definition}

\begin{proposition}[\cite{MP}] A local diffeomorphism $f: M \rightarrow M$ is an Anosov endomorphism of $M$ if and only if the lift $\overline{f}: \overline{M} \rightarrow \overline{M}$ is an Anosov diffeomorphism of $\overline{M},$ the universal cover of $M.$
\end{proposition}

Sakai, in \cite{SA} proved that, in fact, the definitions $\ref{defprz}$ and $\ref{defmp}$ are equivalent. The definition \ref{defmp} will be particularly important for the proof of Theorem \ref{teo3}.

An advantage to work with the definition given in \cite{MP} is that in $\overline{M}$ we can construct invariant foliations $\mathcal{F}^s_{\overline{f}}$ and $\mathcal{F}^u_{\overline{f}}.$ In Theorem \ref{teo3}, we will borrow the transverse structure of  $\mathcal{F}^s_{\overline{f}}$ and $\mathcal{F}^u_{\overline{f}}.$

Let $f: M  \rightarrow M$ be a $C^r-$Anosov endomorphism with $r \geq 1,$ it is know that $E^s$ and $E^u$ are integrable to $C^r-$leaves $W^s_f(\tilde{x})$ and $W^u_f(\tilde{x}),$ which are $C^r-$submanifols, such that
\begin{enumerate}
\item $W^s_f(x) = \{y \in M \;|  \displaystyle\lim_{n \rightarrow +\infty} d(f^n(x), f^n(y)) = 0\},$
\item $W^u_f(\tilde{x}) = \{y \in M \;| \exists \tilde{y} \in M^f \; \mbox{such that}\; y_0 = y \; \mbox{and} \;  \displaystyle\lim_{n \rightarrow +\infty} d(x_{-n}, y_{-n}) = 0\}.$
\end{enumerate}

The leaves $W^s_f(\tilde{x})$ and $W^u_f(\tilde{x})$  vary  $C^1-$continuously with $\tilde{x},$ see Theorem 2.5 of \cite{PRZ}.

Given an Anosov endomorphism let $E^u(\tilde{x})$ denotes the bundle $E^u_{x_0}.$ An Anosov endomorphism for which $E^u(\x)$ just depends on $x_0$ (unique unstable direction for each point) is called special Anosov endomorphism. A linear Anosov endomorphism of the torus is an example of a special Anosov endomorphism. Of course, when $f$ is an special Anosov endomorphism we have $W^u_f(\x) =  W^u_f(\tilde{y}),$ for any $\tilde{x}, \tilde{y}$ such that $x_0 = y_0.$ So makes sense denote in this case $W^u_f(\x) = W^u_f(x_0).$

A contrast between Anosov diffeomorphisms and Anosov endomorphisms is the non-structural stability of the latter. Indeed, $C^1-$close to any linear Anosov endomorphism $A$ of the torus, Przytycki \cite{PRZ} constructed Anosov endomorphism which has infinitely many unstable directions for some positive orbit, and consequently, he showed that $A$ is not structurally stable. However, it is curious to observe that the topological entropy is locally constant among Anosov endomorphisms. Indeed, take the lift of Anosov endomorphism to the inverse limit space (see preliminaries for the definition). At the level of inverse limit space, two nearby Anosov endomorphisms are conjugate (\cite{PRZ}, \cite{BerRov}), and lifting to inverse limit space does not change the entropy.

\begin{definition} A continuous surjection $f:\mathbb{T}^n \rightarrow \mathbb{T}^n $ is said strongly special Anosov endomorphism map if $f$ is a special Anosov endomorphism which is not injective and for each point $x \in \mathbb{T}^n,$ the stable leaf $W^s_f(x)$ is dense in $\mathbb{T}^n.$
\end{definition}

The celebrated theory due Franks, Manning and Newhouse asserts that given $f: \mathbb{T}^n \rightarrow \mathbb{T}^n $ an Anosov diffeomorphism with codimension one, then $f$ is conjugated with its linearization $A.$ It means that there is a homeomorphism $h: \mathbb{T}^n \rightarrow \mathbb{T}^n$ such that $$h\circ A =  f\circ h.$$
We understand the linearization $A$ of $f$ being the action on $\mathbb{T}^n,$ of the matrix with integer entries $A,$  where $A$ is  given by the action of $f$ in $\Pi_1(\mathbb{T}^n) = \mathbb{Z}^n.$

In 1990 years, R. de la Llave in several works characterized the smooth Anosov diffeomorphisms $f: \mathbb{T}^2 \rightarrow \mathbb{T}^2$ which are differentiable conjugated with its linearization $A.$ If fact, the condition is known by the same periodic data between corresponding points, it means that if $p$ and $q$ are periodic points for $A$ and $f$ respectively, with period $n$ and the conjugacy $h$ is such that $h(p) = q, $ then the Lyapunov exponents of $p$ and $q$ coincides, i.e,
$$\displaystyle\lim_{n \rightarrow +\infty}\frac{1}{n} \log(Df^n(q)|E^{\ast}_f(q)) = \displaystyle\lim_{n \rightarrow +\infty} \frac{1}{n} \log(DA^n(q)|E^{\ast}_A(p)), \ast \in {s,u}.$$
Recent advances are made for Anosov diffeomorphisms of $\mathbb{T}^3,$ see \cite{GoGu}. For $n \geq 4,$ there are counterexamples, see \cite{Llave92}.

Of course, a non-special Anosov endomorphism can not be conjugated with its linearization. For strongly special Anosov endomorphism we have the Theorem due to \cite{AH}, that we present shortly.

\begin{theorem}
Every strongly special Anosov endomorphism $f:\mathbb{T}^n \rightarrow \mathbb{T}^n$ is conjugated with its linearization.
\end{theorem}

Our first result relates the regularity of the conjugacy and the assumption of matching corresponding SRB measures by the conjugacy.

\begin{maintheorem} \label{teo3}Let $f,g : \mathbb{T}^2 \rightarrow \mathbb{T}^2$ be two $C^{\infty}-$ Anosov endomorphism with degree $k \geq 1,$ such that their linearizations are equal to  $A,$  a strongly special linear Anosov endomorphism. Suppose that $f$ and $g$ are conjugated by $h,$ such that $h \circ f = g \circ h$ and $h_{\ast}$ sends the SRB and inverse SRB of $f$ to the corresponding analogous measures of $g.$ Then $f$ and $g$ are smoothly conjugated.
\end{maintheorem}

The definitions of SRB and inverse SRB measures are given in the preliminaries section.

We note that in Theorem \ref{teo3} we are not supposing that the endomorphisms are special. In fact, the more interesting case here is the one when the endomorphisms are not special. Theorem \ref{teo3} is the core to obtain smoothness in the next theorem.

\begin{maintheorem}\label{teo1} Let $f: \mathbb{T}^2 \rightarrow \mathbb{T}^2 $ be a $C^{\infty}$ strongly special Anosov endomorphism and $A: \mathbb{T}^2 \rightarrow \mathbb{T}^2,$ its linearization. If the corresponding periodic points of $f$ and $A $ have the same Lyapunov exponents, then $f$ and $A$ are smoothly conjugated.
\end{maintheorem}

Let me clarify that when one considers studying rigidity, necessary and sufficient conditions to get differentiable conjugacy, in general, we handle with $C^1-$meager conditions, as meager as the set of $C^1$ map conjugated to a linear, as in our case. In general, it is not expected rigidity for most (generic) cases.

In a recent preprint \cite{shi}, the authors proved that $f: \mathbb{T}^2 \rightarrow \mathbb{T}^2$ is a strong special Anosov endomorphism if and only if $f$ and its linearization $A$ have the same stable periodic data, and in this case, the conjugacy is $C^{1+ \alpha}$ along stable leaves. Relying on this result our Theorem \ref{teo1} could be stated requiring only periodic data conditions. Since the preprint \cite{shi} is posterior to our work we prefer to keep the original format.

In the same direction, results on the rigidity of Anosov endomorphisms under the assumption of the regularity of foliations can be found in \cite{CV}. In this work, among other things, the authors obtain smooth conjugacy with the linearization assuming the UBD condition. For UBD condition we refer \cite{MT13}.

\begin{maintheorem}\label{teo2} Let $f: \mathbb{T}^2 \rightarrow \mathbb{T}^2 $ be a $C^{\infty}$ strongly special Anosov endomorphism. Suppose that for any $x \in \mathbb{T}^2$ are defined all Lyapunov exponents. Then $f$ is smoothly conjugated with its linearization $A. $
\end{maintheorem}

\begin{question}
In \cite{mic} is proved that if $f: M \rightarrow M$ is a $C^r, r \geq 2,$ Anosov diffeomorphism such that every $x \in M$ is regular for $f,$ then $f$ is transitive. The same statement is true for Anosov endomorphisms.
\end{question}


\section{Comments on the proofs}

In the proof of Theorem \ref{teo3}, since we are not supposing special Anosov endomorphism, we borrow the transverse foliations structure of stable and unstable manifolds of the lift of $f$ on $\mathbb{R}^2.$ We use the well-established SRB theory \cite{QZ,PDL2} for endomorphisms after applying an O.D.E argument, similar to one done in \cite{Llave92}. Finally, applying Journ\'{e}'s Lemma \cite{Journe}, we conclude that $h$ is smooth.

In the proof of Theorems \ref{teo1} we use  Livsic's Theorem to construct, via conformal metrics on leaves. Using an isometric map between corresponding leaves, we conclude that the conjugacy applies invariant leaves of $A$ to corresponding invariant leaves of $f.$ Finally, by Livsic's Theorem, the conjugacy $h$ is $C^{1+ \alpha}$ for some $\alpha > 0.$ So we apply Theorem \ref{teo3} to get smoothness.

To prove Theorem \ref{teo2} we use the specification to ensure that $f$ has constant periodic data. Using again SRB theory, Ruelle's inequality, and Pesin formula we conclude that $f$ and its linearization $A$ have the same periodic data, so we finalize by applying Theorem \ref{teo1}. For similar results on diffeomorphism setting, we refer \cite{LM20}.

In the appendix, we state and develop the proof of Theorem \ref{teo4}. In that proof, we use similar ideas to prove Theorem \ref{teo1} and the steps of \cite{Go} to get the matching of foliations.

\section{Preliminaries on SRB measures for endomorphisms} At this moment we need to work with the concept of SRB measures for endomorphisms. In fact, SRB measures play an important role in the ergodic theory of differentiable dynamical systems. For $C^{1+\alpha}-$systems these measures can be characterized as ones that realize the Pesin Formula or equivalently the measures for which the conditional measures are absolutely continuous w.r.t. Lebesgue restricted to local stable/unstable manifolds. We go to focus our attention on the endomorphism case. Before proceeding with the proof let us give important and useful definitions and results concerning SRB measures for endomorphisms.

First, let us recall an important result.

\begin{theorem}[\cite{QXZ}] Let $(M,d)$ be a compact metric space and $f: M \rightarrow M$ a continuous map. If $\mu$ is an $f-$invariant Borelian probability measure, the exist a unique $\tilde{f}-$invariant borelian probability measure $\tilde{\mu}$ on $M^f,$ such that $\mu(B) = \tilde{\mu}(p^{-1}(B)).$
\end{theorem}

\begin{definition}A measurable partition $\eta$ of $M^f$ is said to be subordinate to
$W^u-$manifolds of a system $(f, \mu)$ if for $\tilde{\mu}$-a.e. $\tilde{x} \in M^f,$ the atom $\eta(\tilde{x}),$ containing $\tilde{x},$ has the following properties:
\begin{enumerate}
\item  $p| \eta(\tilde{x}) \rightarrow p(\eta(\tilde{x}))$ is bijective;
\item  There exists a $k(\tilde{x})-$dimensional $C^1-$embedded submanifold $W(\tilde{x})$ of $M$ such that $W(\tilde{x}) \subset W^u(\tilde{x}),$
$$p(\eta(\tilde{x})) \subset W(\tilde{x})$$
and $p(\eta(\tilde{x}))$ contains an open neighborhood of $x_0 $ in $W(\tilde{x}).$ This neighborhood
being taken in the topology of $W(\tilde{x})$ as a submanifold of $M.$
\end{enumerate}
\end{definition}

We observe that by Proposition 3.2 of \cite{QZ}, such partition can be taken increasing, that means $\eta$ refines $\tilde{f}(\eta) .$ Particularly $ p(\eta(\tilde{f}(\tilde{x}))) \subset p(\tilde{f}(\eta(\tilde{x}))) .$

\begin{definition} Let $f: M \rightarrow M$ be a $C^2-$endomorphism preserving an invariant borelian probability $\nu.$
We say that $\nu$ has SRB property if for every measurable partition $\eta $ of $M^f$ subordinate
to $W^u-$manifolds of $f$  with respect to $\nu$, we have $p(\tilde{\nu}_{\eta{(\tilde{x})}})  \ll m^u_{p(\eta(\tilde{x}))},$  for $\tilde{\nu}-$a.e. $\tilde{x}$, where
$\{\tilde{\nu}_{\eta{(\tilde{x})}} \}_{\tilde{x} \in M^f}$
is a canonical system of conditional measures of $\tilde{\nu}$ associated with $\eta,$
and $m^u_{p(\eta(\tilde{x}))} $ is the Lebesgue measure on $W(\tilde{x})$ induced by its inherited Riemannian metric as a submanifold of $M.$
\end{definition}

 In the case of above definition, if we denote by $\rho^u_f$ the densities of conditional measures $\tilde{\nu}_{\eta({\tilde{x}})},$ we have
 \begin{equation} \label{conditionalU}
 \rho^u_f(\tilde{y}) =
\frac{\Delta^u_f(\tilde{x}, \tilde{y} )}{L(\tilde{x})},
 \end{equation}
for each $\tilde{y} \in \eta({\tilde{x}}),$ where
$$ \Delta^u_f(\tilde{x},\tilde{y}) = \displaystyle\prod_{k=1}^{\infty} \frac{J^uf(x_{-k})}{J^uf(y_{-k})}, \tilde{x} = (x_k)_{k \in \mathbb{Z}}, \tilde{y} = (y_k)_{k \in \mathbb{Z}}  $$
and
$$L(\tilde{x}) = \int_{\eta(\tilde{x})} \Delta^u_f(\tilde{x}, \tilde{y}) d \tilde{m}^u_{\eta({\tilde{x}})}(\tilde{y}).$$

The measure $\tilde{m}^u_{\eta({\tilde{x}})}$ is such that $p(\tilde{m}^u_{\eta({\tilde{x}})})(B) = m^u_{p(\eta({\tilde{x}}))}(B).$  Therefore $$p(\tilde{\nu}_{\eta({\tilde{x}})}) \ll m^u_{p(\eta({\tilde{x}}))},  $$ and
$$\rho^u_f(y) =
\frac{\Delta^u_f(\tilde{x}, \tilde{y})}{L(\tilde{x})}, y \in p(\eta({\tilde{x}})).$$

\begin{theorem}{\cite{PDL1}}\label{pesin1} Let $f : M \rightarrow M$ be a $C^2$ endomorphism and $\mu$ an
$f-$invariant Borel probability measure on $M.$ If $\mu \ll m,$ then there holds
Pesin's formula
\begin{equation}
h_{\mu}(f) =  \displaystyle\int_M \displaystyle\sum \lambda^i(x)^{+}m_i(x) d\mu.
\end{equation}
\end{theorem}

\begin{theorem}[\cite{QZ}] \label{pesin2} Let $f$ be a $C^2$ endomorphism on $M$  with an invariant Borel probability
measure $\mu$ such that $\log(|Jf(x)|) \in L^1(M,\mu).$ Then the entropy
formula
\begin{equation}\label{PesinU}
h_{\mu}(f) =  \displaystyle\int_M \displaystyle\sum \lambda^i(x)^{+}m_i(x) d\mu
\end{equation}
holds if and only if $\mu$ has SRB property.
\end{theorem}

%

There are analogous formulations concerning subordinate partition with respect to stable manifolds, which can be take decreasing, that means $f^{-1} (\eta) \preceq \eta,$ see \cite{PDL2}, Proposition 4.1.1.   In the sense of hyperbolic repellors, including  Anosov endomorphisms, there is an important result concerning inverse SRB measures.

\begin{theorem}\label{teonegativo}[Theorem 3 of \cite{Mh1} and Theorems 2.3 and 2.6 of \cite{PDL2}] Let $\Lambda$ be a connected hyperbolic repellor for a smooth $f: M \rightarrow M .$ Assume that $f$ is $d$ to one, then there is a unique $f-$invariant probability measure $\mu^{-}$ on $\Lambda$ satisfying the inverse Pesin formula

\begin{equation}\label{PesinS}
h_{\mu^{-}}(f) = \log(d) - \displaystyle\int_M \displaystyle\sum \lambda^i(x)^{-}m_i(x) d\mu^{-}.
\end{equation}
In addition, the measure $\mu^{-}$ is characterized by having absolutely continuous conditional measures on local stable
manifolds.

\end{theorem}

In the setting of the previous Theorem, if $(f, \mu)$ satisfies the Stable Pesin Formula \ref{PesinS}, then for a given subordinate partition $\eta,$ with respect to stable manifolds, we have $$ \mu_{\eta (x)} \ll m^s_{\eta (x)},$$ for $\mu-$ a.e $x \in M.$ Moreover
\begin{equation}\label{conditionalS}
\rho^s_f(x) = \frac{\Delta^s_f(x,y)}{\int_{\eta(x)} \Delta^s_f(x,y)dm^s_{\eta (x)} }, \; \forall y \in \eta(x).
\end{equation}

Here $\Delta^s_f(x,y) = \prod_{k = 0}^{\infty} \frac{Jf(f^k(x))}{Jf(f^k(y))}\cdot \frac{J^sf(f^k(x))}{J^sf(f^k(y))}.$
See \cite{PDL2} as a reference.

The theorems on Pesin formulas are true in our setting since every tori Anosov endomorphism is transitive, see \cite{AH}.

We finalize the preliminaries section with a lemma whose proof is essentially the same as Corollary 4.4 of \cite{Llave92}, up to minor adjustments using local inverses.

\begin{lemma} For a $C^{k}, k \geq 2,$ Anosov endomorphism, the conditional measures of stable and unstable SRB measures restricted to stable and unstable leaves respectively are $C^{k-1}.$ In particular, if $f$ is smooth, then the conditional measures are smooth.
\end{lemma}

\section{Proof of Theorem \ref{teo3}}

\begin{lemma}\label{arc} Consider $f$ as Theorem \ref{teo3}, then given $V$ a $s-$foliated neighborhood there exist $R > 0,$  such that every stable arc with size bigger than $R$ crosses $V.$ An analogous statement holds for unstable leaves.
\end{lemma}

\begin{proof}
First, since $f$ is continuous, it is not hard to see that a set $D$ is dense in $\mathbb{T}^2$ if and only if $p^{-1}(D)$ is dense in the limit inverse space $M^f = (\mathbb{T}^2)^f.$ It is known by  \cite{AH} that $f$ and $A$ are conjugated in the limit inverse level. Since $A$ has dense stable leaves we conclude that all stable leaf of $f$ is also dense. Given $x \in M$ the leaf $W^s_f(x)$ is dense in $M,$ then there is $R(x) > 0$ such that any stable arc starting in $x$  with size $ R \geq R(x)$ crosses $V.$ By continuity of stable manifold there is a neighborhood $B(x) \in M$ such that $z \in B(x),$ then any stable arc starting in $z$ with size $2R(x)$ crosses $V.$ By compactness of $M$ there is a finite cover $B(x_1), \ldots, B(x_k)$ of $M.$ Choose $R = \max\{ 2R(x_i)\}, i=1,\ldots, k.$
\end{proof}

It is know that given $f:M \rightarrow M$ an Anosov endomorphism, its lift $\bar{f}: \overline{M} \rightarrow \overline{M} $ is an Anosov diffeomorphism. For $\overline{f}$ makes sense unstable and stable invariant foliations of $\overline{M}.$ Locally we can use the natural projection $\pi: \overline{M} \rightarrow M,$ to consider in $M$ locally unstable and stable foliations of $f. $ We need to prove that the conjugacy $h$ between $f$ and $g$ as in Theorem \ref{teo3} is smooth restricted to each local leaf projected by $\pi$ and conclude the result by using Journ\'{e}'s Theorem.

\begin{lemma} Consider $f$ and $g$ as Theorem \ref{teo3}.  Given $z_0 \in M,$  consider $V$ a small neighborhood of $z_0$ foliated by $\mathcal{F}^s_f$ and $\mathcal{F}^u_f$ projected by $\pi: \overline{M} \rightarrow M.$ Then the conjugacy $h$ is smooth restricted to each stable and unstable local leaf in $V.$
\end{lemma}

\begin{proof} First, consider $m_f$ and $m_g$ the respective SRB measures for $f$ and $g,$ then for both holds the Pesin Formula (see Theorem \ref{pesin1} ) and the theory of SRB measures of \cite{QZ}, consequently have absolutely continuous disintegration along unstable manifolds.
For the partition $\eta_k = \tilde{f}^k(\eta),$ where $\eta$ is any subordinate partition w.r.t unstable leaves,  consider the $\widetilde{m_f}-$full measure set of points $X_k,$ of points satisfying $(\ref{conditionalU}).$ Now take $X = \bigcap_{k=0}^{+\infty} X_k,$ and finally $\mathcal{T} = \bigcap_{j=0}^{+\infty} \tilde{f}^{-j} X .$ The projection on $\mathbb{T}^2$ of $\mathcal{T}$ has $m_f-$full measure. So given any $x \in \mathcal{T} ,$ the iterates $f^k(x), k \geq 0,$ satisfies $(\ref{conditionalU})$ for the corresponding projection of $\eta_k = \tilde{f}^k(\eta).$

 Since $h_{\ast}(m_f) = m_g,$ then  $h$ sends conditional measures of $(f, m_f)$ in conditional measures of $(g, m_g).$ Since these measures are equivalent to Riemannian measures of unstable leaves, so $h$ sends null sets of $p(\eta(\tilde{x}))$ in null sets of $p(\eta(\tilde{h}(\tilde{x})))$ with respect to  Riemannian measures of unstable leaves, where $\tilde{h}$ is the conjugacy at level of limit inverse space between $\tilde{f}$ and $\tilde{g}.$

Consider $B^u_{x_0} \subset \eta(\tilde{x}) $ a small open unstable arc. Since $h$ is absolutely continuous
$$\int_{B^u_{x_0}} \rho^u_f(y) dy  = \int_{h({B^u_{x_0}})} \rho^u_g(y) dy = \int_{B^u_{x_0}} \rho^u_g(h(y)) h'(y) dy,  $$
therefore solving the O.D.E.
\begin{equation}\label{ODE1}
x' = \frac{\rho^u_f(t)}{\rho^u_g(x)}, x(x_0) =  h(x_0),
\end{equation}
we find $h$ is $C^{\infty}$ on $B^u_{x_0}.$

Since the unstable leaves are dense in $(\mathbb{T}^2)^f,$ because $\tilde{A}$ and $\tilde{f}$ are conjugate in the limit inverse level. So by denseness and Lemma \ref{arc} we can get a sequence of arcs  $W_n \subset f^n(B^u_{x_0}) \cap V,$ is such that  $W_n \rightarrow_{C^1} \mathcal{F}^u_f(z_0),$ where $\mathcal{F}^u_f(z_0)$ is the local unstable manifold projected of $\overline{M}$ at $z_0$ in $V.$

Since the subordinate partition can be taken increasing, see Proposition 3.2 of \cite{QZ}, the conjugacy $h$ restricted to $W_n$ satisfies an analogous O.D.E, as in $(\ref{ODE1}).$

Normalizing the conditional measures such that $$ \int_{W_n} c_n\cdot \rho^u_f(t) dVol_{W_n} = 1,$$ since $h_{\ast}(\rho^u_f(t) dVol_{W_n}) = \rho^u_g(t) dVol_{h(W_n)},$ then $h$ send normalized conditional measures into normalized conditional measures. For simplicity consider $c_n = 1,$ for each $n$ and consider normalized densities $\rho^u_f$ and $\rho^u_g.$

For the points $y \in W_n,$ take the initial condition $y_0,$ where $y_0$ is an arbitrarily chosen point in $W_n.$ We know $$\rho^u_f(y) = \alpha_n\cdot\Delta^u_f(y_0, y),$$ for some constant $\alpha_n.$ We note that $\alpha_n$ is bounded and far from zero since we size of $W_n$ is uniformly bounded as well as the value of $\Delta^u_f.$ For $g,$ by analogous reason
$$\rho^u_g(y) = \beta_n\cdot\Delta^u_g(y_0, y),$$
$\beta_n$ is bounded and far from zero. Since $\frac{\alpha_n}{\beta_n}$ is positive far from zero and uniformly bounded, for simplicity, we suppose that $\alpha_n = \beta_n.$
In this way, by relation $(\ref{ODE1})$,   $h$ satisfies the following O.D.E,
$$x' = \frac{\Delta^u_f( y_0 , t)}{\Delta^u_g( h(y_0), x)}, x(y_0) =  h(y_0),$$
for each pair of connected component $W_n$ and  $h(W_n).$

Denoting by $h_n$ the solution of the above equation, we note that the solution $h_n$ is smooth. The map $h_n$ is the restriction of the conjugacy $h$ on $W_n.$ Analogous to Lemma 4.3 of \cite{Llave92}, for each component $W_n$ we have a collection $\{h_n: W_n \rightarrow h(W_n)\}_{n = 1}^{\infty},$  is uniform bounded as well the collection of their derivatives of order $r = 1,2,\ldots.$   By an Arzela-Ascoli argument type applied to a sequence $h_n$ and the sequence of their derivatives, we conclude that $h$ is $C^{\infty}$ restricted to $\mathcal{F}^u_f(z_0).$

For stable leaves, we use a similar argument, arguing with inverse SRBs $m_f^{-}$ and $m_g^{-},$ such that $h_{\ast}(m_f^{-}) = m_g^{-}.$ The stable foliation restricted to $V$ is an absolutely continuous foliation, then for $m_f^{-}-$a.e. point $t \in \mathbb{T}^2$  holds \eqref{conditionalS} for any point $y \in \eta(t).$  The connected components of $f^{-n}(B^s_{x_0})$ grows exponentially. By Lemma \ref{arc} we can choose stable arcs in pre-images such that
$$f^{-n}({B^s_{x_0}})\cap V \rightarrow_{C^1} \mathcal{F}^s_f(z_0) $$
in $C^1-$topology.

Relying in the expression $(\ref{conditionalS})$ for a decreasing subordinate partition with respect stable manifolds, as in the argument for unstable leaves, via O.D.E,  $h$ is $C^{\infty}$ restricted to each component of pre images $f^{-n}({B^s_{x_0}}).$ By an Arzela-Ascoli type argument we obtain $h$ is $C^{\infty}$ restricted to $\mathcal{F}^s_f(z_0).$

The same argument can be applied for any point $z  \in V.$
\end{proof}

To finalize the proof of Theorem \ref{teo3}, we evoke the following classic result applied to $h.$

\begin{theorem}[Journ\'{e}'s Theorem] Let $F_s$ and $F_u$ two continuous and transversal foliations with uniformly smooth leaves, of some manifold. If $f$ is uniformly $C^{r + \alpha}, \alpha > 0$ and $r \geq 1$ along the leaves of $F_s$ and $F_u,$ then $f$ is $C^{1+\alpha -\varepsilon},$ for any $\varepsilon > 0.$ Particularly, if $r = \infty,$ we conclude that $f $ is $C^{\infty}.$
\end{theorem}

We conclude that $h$ is smooth.

\section{Proof of Theorem \ref{teo1}}

Here, first, we prove that $f$ and $A$ are $C^1,$ conjugated, so $f$ preserves a measure equivalent to Lebesgue and so we can apply Theorem \ref{teo3} to conclude that $f$ and $A$ are smoothly conjugated. For this, we need some tools to proceed as in \cite{LM20}.

An important tool related to Livsic's Theorem. By \cite{Chung} is known a version of the shadowing lemma for endomorphisms.

\begin{proposition}[Closing Lemma for Endomorphisms, Lemma 3 of \cite{Chung}] For $0\leq k\leq\dim M,$ $\chi>0,$ $l\geq 1$ and $\rho>0$ there exists
a number $\gamma_{l}(\rho)=\gamma_{l}(k, \chi, \rho)>0$ such that, if $\tilde{x}=(x_{n})\in\tilde{\Lambda}_{\chi,l}^{k}$ satisfies
$f^{m}(\tilde{x})\in\tilde{\Lambda}_{\chi,l}^{k},$  $d(f^{m}(\tilde{x}),\tilde{x})\leq\gamma_{l}(p)$
for some $m\geq 1,$  then there is a hyperbolic periodic point $p=p(\tilde{x})\in M$ of$f$ with $f^{m}(p)=p$
such that
$d(f^{j}(p), x_{j})\leq\rho$
for all $0\leq j\leq m-1$ .
\end{proposition}

The point $p$ above is unique. In the context of Anosov endomorphisms, there is a suitable choice of constants such that $\tilde{\Lambda}_{\chi,l}^{k} = M,$ in this specific case the above proposition is known by Anosov Closing Lemma. Endowed with the Anosov Closing Lemma we can prove, using the same argument as the version for diffeomorphisms, the following version of Livsic's Theorem.

\begin{theorem}[Livsic's Theorem] Let $M$ be a Riemannian manifold, $f:M \rightarrow M$ a transitive smooth
Anosov endomorphism and  $\varphi: M  \rightarrow \mathbb{R}$
an $\alpha-$ H\"{o}lder function. Suppose that for every $x \in M$ such that $f^n(x) = x,$ we have
$\displaystyle\sum_{i=0}^{n-1}\varphi(f^i(x)) = 0.$ Then there exists a unique  $\alpha-$ H\"{o}lder function $ \phi:  M  \rightarrow \mathbb{R},$ such that
$\varphi(x) = \phi(f(x))- \phi(x)$ and $\phi$  is unique up to an additive constant.
\end{theorem}

For the proof see \cite{K}, page 610.

Let us introduce conformal distances on each invariant one-dimensional leaf.

\begin{lemma} There exists a metric $d^u$ on each leaf $W^{u}_f(x)$ tangent to $E^{u}_f,$  such that  $d^u(f(a), f(b)) = e^{\lambda^u} d^u(a, b),$ where $\lambda^u$ the common value of the Lyapunov exponents of periodic points of $f$ and $A$ relative to directions $E^u_f$ and $E^u_A$ respectively.
\end{lemma}

\begin{proof}
Denote by $\lambda^u$ the common value of the Lyapunov exponents of periodic points of $f$ and $A$  in the directions $E^{u}_f$ and $E^{u}_A,$  respectively. Let us to denote on $\mathbb{T}^2,$ the $f-$invariant foliations $\mathcal{F}^{\ast}_f$ tangent to $E^{\ast}_f, \ast \in \{s,u\}.$

We see that $\log (||Df(x)|E^{u}_f(x)||) -\lambda^u$ has zero average over every periodic
orbit.

Since $f$ is a $C^{1+ \alpha}-$Anosov diffeomorphism, the map $x \mapsto \log (||Df(x)|E^{u}_f(x)||) $  is uniform $C^{\varepsilon}$ on $\mathbb{T}^2,$ for some $\varepsilon > 0.$ Hence, by Livsic's theorem \cite{Livsic72,Bo74}, we can find a $C^{\varepsilon}-$function $\phi^u$ such that,
$\phi^u: \mathbb{T}^2 \rightarrow \mathbb{R}$ such that
\begin{equation}  \label{cohomology}
\log (||Df(x)|E^{u}_f(x)||) - \lambda^u =    \phi^u(f(x)) - \phi^u(x).
\end{equation}

Equivalently
\begin{equation} \label{conformal}
e^{\phi^u(x)}||Df(x)|E^{u}_f(x)|| e^{-\phi^u(f(x))} = e^{\lambda^u}.
\end{equation}

We can interpret \eqref{conformal} as saying that, if we define a
metric, conformal to the standard metric in the torus by a factor
$e^{-\phi^u},$ then for a convenient metric $f$ expands on $W^{u}_i-$leaves  by exactly
$e^{\lambda^u}.$

In fact, fix an orientation on $W^{u}_f(x)$ and consider $a \geq b$ on $W^{u}_f(x),$ consider the metric
$$d^u(a, b) = \int_a^b e^{-\phi^u(x)}dx, $$ where $dx$ denotes the infinitesimal size on $W^{u}_f(x).$ With this

$$ d^u(f(a), f(b)) = \int_{f(a)}^{f(b)} e^{-\phi^u(y)}dy = \int_a^b e^{-\phi^u(f(x))}||Df(x)|E^{u}_f(x)|| dx =$$ $$= e^{\lambda^u} \int_a^b e^{-\phi^u(x)}dx = e^{\lambda^u} d^u(a,b).$$

\end{proof}

Also, we need the following proposition.

\begin{proposition}[Proposition 8.2.2 of \cite{AH}]\label{propAH} Let $L : \mathbb{R}^n \rightarrow \mathbb{R}^n$ be a hyperbolic linear automorphism
and let $T : \mathbb{R}^n \rightarrow \mathbb{R}^n$ be a homeomorphism. If $\bar{d}( L, T)$ is finite, then there is a
unique map $\phi : \mathbb{R}^n \rightarrow \mathbb{R}^n$ such that
\begin{enumerate}
\item $L \circ \phi = \phi \circ T,$
\item $\bar{d}(\phi, id_{\mathbb{R}^n})$ is finite.

\hspace{-1.7cm}Furthermore, for $K > 0$ there is a constant $\delta_K > 0$ such that if $\bar{d}(L,T) < K,$

\hspace{-1.7cm}then the above map $\phi$ has the following properties :

\item $\bar{d}(\phi, id_{\mathbb{R}^n}) < \delta_K,$
\item  $\phi$ is a continuous surjection,
\item  $\phi$ is uniformly continuous under $\bar{d}$ if so is T.
\end{enumerate}
\end{proposition}

\begin{proof}

We go to prove the differentiability of the conjugacy between $f$ and $L,$ by using the conformal metrics on each one-dimensional invariant foliation of $f.$

Let $h: \mathbb{T}^2 \rightarrow \mathbb{T}^2$ be the conjugacy between $f$ and $A,$ such that $$h \circ A = f \circ h.$$

We first observe that, since $h$ sends $W^u_A$ leaves in $W^{u}_f,$ leaves then $h$ induces naturally a conjugacy $\mathcal{H}: \mathbb{T}^2/\mathcal{F}^{u}_A \rightarrow \mathbb{T}^2/\mathcal{F}^{u}_f.$

Up to change $(A, f)$ by $(A^2, f^2)$ we can suppose  $A$ and $f$ preserve the orientations established.

 Using this orientation, choose points $a_j , j \in \mathbb{Z}$ such that $a_{j} < a_{j+1}$ and for simplicity suppose that $|a_j - a_{j+1}| = 1,$ where $|u - v|$ is the Euclidean distance induced on $W.$ In fact we are seeing $W$ as a real line. Let $b_j = h(a_j), j \in \mathbb{Z}.$ For each $j$ we choose a function $\phi^u_{j}$ such that $d^u$ is such that $d^u(b_j, b_{j+1}) = 1.$ To simplify the writing,  we denote by $[p, q]$ a segment connecting points $p$ and $q$ on a leaf of type $W^{u}_A$ and $W^{u}_f.$ The same notation we will use for leaves lifted on $\mathbb{R}^2.$

Let us to define a map $\tilde{h}: [a_j, a_{j+1}] \rightarrow [b_j, b_{j+1}], $ using $\phi^u_{j}$ and the corresponding $d^u$ such that $\tilde{h}(\theta)$  is the unique point $p$ in $[b_j, b_{j+1}]$ such that $d^u(b_j , p ) = |a_j - \theta|.$  Also, for the given $j,$ using $\phi^u_{j}$ and the corresponding $d^u$ we define $\tilde{h}: [A^n (a_j), A^n(a_{j+1})] \rightarrow [f^n(b_j), f^n(b_{j+1})] $ following the same strategy before, for each $n \in \mathbb{Z}.$ By construction, $\tilde{h}$ and $h$ coincide on the extremes of intervals, as defined. This construction is such that $\tilde{h} \circ A = f \circ \tilde{h}.$ In fact, consider $\theta \in [a_0, a_1]$ such that $|a - \theta | = \alpha.$ By definition $d^u (\tilde{h}(a_0), \tilde{h}(\theta)) = \alpha.$ Taking the first iterated $|A(a_0) - A(\theta)| = e^{\lambda^u}\alpha$ and $d^u(f(\tilde{h}(a_0)) , f(\tilde{h}(\theta))) = e^{\lambda^u} d^u (\tilde{h}(a_0), \tilde{h}(\theta)) = e^{\lambda^u} \alpha.$ By definition $f (\tilde{h} (\theta)) = \tilde{h}(A(\theta)).$ The analogous construction can be done in universal cover level using lifts $\bar{A}$ and $\bar{f}.$

We can describe $\tilde{h}$ as a solution of a specific ordinary differential equation. In fact, given a leaf $W = W^{u}_A,$  $\tilde{h}:[a_0, a_1] \rightarrow [b_0, b_1]$ is defined by
\begin{equation}\label{ODE}
z' = e^{\phi^u_{0}(z)}, z(a_0) = b_0.
\end{equation}

In fact, let $z:[a_0,a_1] \rightarrow [b_0, b_1]$ be a solution of the differential equation $(\ref{ODE}).$ Let $a_0 \leq \theta \leq a_1,$ we have
$z'(t)e^{-\phi^u_{0}(z(t))} = 1,  $ for any $t \in [a_0, a_1],$ so

$$ \theta - a_0 = \int_{a_0}^{\theta} e^{-\phi^u_{0}(z(t))} z'(t) dt =  \int_{z(a_0)}^{z(\theta)} e^{-\phi^u_{0}(s)} ds  = d^u(z(a_0),z(\theta) ) = d^u(b_0,z(\theta) ) ,   $$
here $ds$ denote  the infinitesimal length arc of $W^{u}_f(b_0),$  so $z(\theta) = \tilde{h}(\theta),$ by definition of $\tilde{h}.$ In consequence, differential equations of kind $(\ref{ODE})$ have unique solution. In particular $\tilde{h}$ is at least $C^{1 + \alpha},$ for some $\alpha > 0,$ on each interval, since the functions $\phi$ is at least Lipschitz, given by Livsic's Theorem.

We will use these ideas in the universal cover. Consider $\bar{W}$ a lift of an unstable leaf $W,$ and $\bar{H}$ the lift of $h.$ Consider on $\bar{W}$ a the collection of points $\{ u_i \}_{i \in \mathbb{Z}}$ corresponding the intersection of $\bar{W}$ with the boundary of fundamental domains $[0,1]^2 + n, n \in \mathbb{Z}^2.$ We consider $\{z_i\}$ the collections of pre-images of the points $x,$ such that $\bar{A}(\bar{W})$ crosses fundamental domains at $x.$ To simplify, since $\bar{W}$ is ordered (induced by the order in $W$) we call $\{a_i\}$ the union of both collection, and $b_i = H(a_i).$ Note that the collection $\{b_i\}$ related to fundamental domains $H([0,1]^2 + n)$ have analogous properties described for  $\{a_i\}.$ So we define $\tilde{H}: \bar{W} \rightarrow H(\bar{W}) = \bar{f}(\bar{W}), $ such that $\tilde{H}(a_i) = b_i,$ and restricted to each interval $[a_i , a_{i+1}],$ the map $\tilde{H}$ is defined using in the universal cover level the metrics $d^u,$ making $\tilde{H}$ an isometry between $([a_i, a_{i+1}], |\cdot| )$  and $([b_i, b_{i+1}], d^u).$

Note that, if $\gamma_i$ is the connected component of $\bar{W}$ inside a fundamental domain $D_i$ of kind $[0,1]^2 + n, $ by construction $\tilde{H}(\gamma_i) = H(\gamma_i) = \delta_i,$ where $\delta_i$ is the connected component of $H(\bar{W})$ inside a fundamental domain $H(D_i).$

The reason to choose points $z_i$ is that $\bar{A}(z_i)$ are exactly the points such that $\bar{A}(\bar{W})$ crosses fundamental domains, it allows one more time define $\tilde{H}$ preserving connected components inside fundamental domains. So over $\bar{A}(\bar{W})$ we consider points the union of points $\bar{A}(a_i)$ and the analogous  $z_i$ for $\bar{A}(\bar{W}),$ so we can proceed inductively the construction of $\tilde{H}$ on each $\bar{A}^k(\bar{W}), k \in \mathbb{Z},$ since $\bar{A}$ is invertible. We repeat this process for all orbits of unstable leaves. 

%

Since we are deal with a foliation and  $\tilde{H}$ is bijective restricted to each leaf, we get a bijection $\tilde{H}: \mathbb{R}^2  \rightarrow \mathbb{R}^2,$ such that, as above $\tilde{H}(\gamma_i) = H(\gamma_i)$ and $\tilde{H} \circ \bar{A} = \bar{f} \circ \tilde{H}.$ So there is $K > 0$ such that

\begin{equation} \label{distance2}
x \in \mathbb{R}^2 \Rightarrow || H(x) - \tilde{H}(x)|| \leq K .
\end{equation}

Finally, since $H$ is the lift of $h,$ we get $|| H(x) - x|| \leq R,$ for any $x \in \mathbb{R}^2$ and we conclude

\begin{equation} \label{distance3}
x \in \mathbb{R}^2 \Rightarrow ||  \tilde{H}(x) - x|| \leq R + K .
\end{equation}

By Proposition $\ref{propAH}$ we conclude $H = \tilde{H},$ and then $H$ is at least $C^{1+\varepsilon}$ on unstable leaves unless countable points per unstable leaves. We could to do it again using fundamental domains $([0,1]^2 + \vec{\varepsilon}) + n, n \in \mathbb{Z}^2,$ where $\vec{\varepsilon}$ is a small vector in $\mathbb{R}^2,$ with irrational coordinates. The conclusion for these constructions is that $H$ is differentiable up to countable points on unstable leaves corresponding to crosses with fundamental domains, but this set is disjoint to the set in the before situation, so $H$ is $C^{1+\varepsilon}$ on unstable leaves.

To finalize the argument that $h$ is at least $C^1$ we evoke again the Journ\'{e}'s Lemma since the argument can be applied for stable leaves. Finally $h$ is $C^1,$ we conclude that $h$ is smooth by applying Theorem \ref{teo3} to $f$ and $A.$

\end{proof}

\begin{remark}
The technique used to prove the regularity of the conjugacy in Theorem \ref{teo1} can be used in every context such we have the coincidence of periodic data along the corresponding one-dimensional foliations of $f$ and $A$ and $h$ matches leaves of such foliations.
\end{remark}


\section{Proof of Theorem \ref{teo2}}
For this section, we need the specification to prove the following lemma.

\begin{lemma}\label{spc} Consider $f: \mathbb{T}^2 \rightarrow  \mathbb{T}^2 $ an Anosov endomorphism such that every point is regular. So for any point $p,q  \in Per(f)$ holds
$$\lambda^{\ast}_f(p) = \lambda^{\ast}_f(q), \ast \in \{s, u\}.$$
\end{lemma}

We present the proof later.

As in equation \eqref{conformal},  $$||Df^n(x)|E^s_f(x)|| = e^{n\lambda^s} e^{\phi_s(f^n(x)) - \phi_s(x)},$$ for some $\phi_s: \mathbb{T}^2 \rightarrow \mathbb{R}$ a Lipschtiz function. So  $\lambda^s_f(x) = \lambda^s,$ for any $x \in \mathbb{T}^2.$ Since $\phi_s$ is continuous, the convergence $\frac{1}{n}\log(||Df^n(x)|E^s_f(x)||) \rightarrow \lambda^s$ is uniform on $\mathbb{T}^2.$ Analogously $\lambda^u_f(x) = \lambda^u,$ for any $x \in \mathbb{T}^2,$ with uniform convergence. The same idea holds for $Jf =  |\det(Df)|,$ meaning that there is a Lipschitz function $\phi: \mathbb{T}^2 \rightarrow \mathbb{R}, $ such that
\begin{equation}\label{jac}
Jf = e^{c} e^{\phi(f(x)) - \phi(x)}.
\end{equation}
By Oseledec's Theorem $c = \lambda^u + \lambda^s$ in $\eqref{jac}.$ Of course, the convergence $\frac{1}{n}\log(Jf(x)) \rightarrow \lambda^u +\lambda^s$ is uniform.

%

\begin{lemma} Consider $f: \mathbb{T}^2 \rightarrow  \mathbb{T}^2$ an Anosov endomorphism such that every point is regular. Then $f$ preserves an absolutely continuous measure.
\end{lemma}

\begin{proof}
%

Since $Jf = |\det Df|$ is cohomologous to constant, by \cite{Mh} we obtain $h_{\mu^{+}_f}(f) =  h_{\mu^{-}_f}(f),$ we conclude that  $\lambda^u + \lambda^s = \log(k).$ The formula $\eqref{jac}$ can be rewritten as

\begin{equation}\label{density}
\log(J f) - \log(k) =  \phi( f(x)) - \phi(x).
\end{equation}

It leads us to  $$Jf(x) e^{-\phi(f(x))} = ke^{-\phi(x)}. $$
Define

Let $B$ be a small open ball and $B_1, B_2, \ldots, B_k$ its mutually disjoint pre images, $f(B_i) = B.$

Define the measure $d\nu = e^{-\phi(x)}dm,$
$$\nu(B) = \nu(f(B_i)) = \int_{f(B_i)} e^{-\phi(y)}dm  =  \int_{B_i} Jf(x) e^{-\phi(f(x))}dm = \int_{B_i} k e^{-\phi(x)}dm = k\nu(B_i) $$
$$\nu(B_i) = \frac{1}{k} \nu(B)$$

$$ \nu(B)  = \sum_{i = 1}^k\nu(B_i) = \nu(f^{-1}(B) ).$$
Define
$\mu(X) = \frac{\nu(X)}{\nu(\mathbb{T}^2)},$ to obtain an $f-$invariant measure absolutely continuous w.r.t. $m.$
\end{proof}

Let us end the proof of Theorem \ref{teo2}. We know that $\lambda^u_f(x) = \lambda^u,$ for any $x \in \mathbb{T}^2.$ Using the Ruelle's inequality we obtain $$h_{\nu}(f) \leq \lambda^u,$$ for any $\nu$ an $f-$invariant, Borelian, probability measure. By variational principle $$h_{top}(f) \leq \lambda^u.$$

By the version of the Pesin Theorem for endomorphism,
$$h_{\nu}(f) = \log(\lambda^u).$$ So $\nu = \mu^{+}_f = \mu^{-}_f$ the maximal entropy measure of $f. $
Since $f$ and $A$ are conjugated, they are same topological entropy, then $\lambda^u = \lambda^u_A$ and $\lambda^s = \lambda^s_A,$ using Theorem \ref{teo1} we conclude the proof.

\section{Specification Property and Proof of Lemma \ref{spc}}

Let us explain the specification property.

\begin{definition}[Specification Property] Let $f: M \rightarrow M$ be a diffeomorphism. We say that $f$ has the specification property if given $\varepsilon > 0$ there is a relaxation time $N\in \mathbb{N}$  such that every $N-$spaced collection of orbit segments is $\varepsilon-$shadowed by an actual orbit. More precisely, for points $x_1, x_2, \ldots, x_n$ and lengths $k_1, \ldots, k_n \in \mathbb{N}$ one can find times $a_1, \ldots, a_n$ such that $a_{i+1} \leq a_i + N$ and a point $x$ such that $d(f^{a_i + j}(x), f^{j}(x_i) ) < \varepsilon $ whenever $0 \leq j \leq k_i.$ Moreover, one can choose $x$ a periodic point with period no more than $a_n + k_n + N.$
\end{definition}

\begin{theorem}[Bowen, \cite{Bo74}] Every transitive Anosov diffeomorphism has the specification property.
\end{theorem}

Recently Moriyasu, Sakai, and Yamamoto in \cite{SY} proved among other things the following result.

\begin{proposition}[Corollary 1 of \cite{SY}] The set of $C^1$-regular maps of $M$ satisfying the $C^1$-stable specification
property is characterized as the set of transitive Anosov maps.
\end{proposition}

So we can apply specification to sketch prove Lemma \ref{spc}.

\begin{proof}
Suppose that $p$ and
$q$ periodic points of $f$ such that $f^n(p) = p$ and $f^n(q) = q,$
where $n \geq 1$ is an integer number. Suppose that
$\lambda(p),\lambda(q)$ denote the Lyapunov exponents corresponding to
direction $E^u_f$ and $\lambda(p) < \lambda(q).$ Consider $\delta > 0$
such that $(1+\delta)^2\lambda(p) < (1-\delta^2)\lambda(q),$ and
$\varepsilon > 0$ such that if $d(x,y) < \varepsilon, $ then $1-\delta
< \frac{|D^uf(x)|}{|D^uf(y)|} < 1+ \delta.$ Let $N > 0$ be the
relaxation time, for the given $\varepsilon > 0,$ where $D^uf(x) =
Df(x)|E^u_f(x).$ For each $j \in \mathbb{N}$ we consider the orbit
segments $P_j = \{\theta_j, f^1(\theta_j), \ldots, f^{k_j - 1}
(\theta_j)\}, $ where $\theta_j = p, $ if $j$ is odd and $\theta_j =
q, $ if $j$ is even. We define inductively $k_j$ as follows.  First
$k_1 = n,$ $k_{j+1} = (k_1 + \ldots + k_j + jN)^2,$ for $j = 1, 2,
\ldots .$ Consider $O_j $ the concatenation of $P_1, \ldots, P_k .$
The length of the sequence $O_j$ is $k_1 + \ldots + k_j.$ By
specification property of $f,$ for any $j $ there is a point $z_{j}$ and a segment of orbit $\{z_{j}, f(z_{j}), \ldots,
f^{r_{j}}(z_{j})\},$ with $r_{j} \leq (k_1 + \ldots + k_{j-1} +
(j-1)N) + k_{j}$ and satisfying the specification property.  Observe that
$r_{j}$ is a natural number of the form $s_j + t_j^2,$ with $t_j =
(k_1 + \ldots + k_{j-1} + (j-1)N)\in \mathbb{N}$ and $0< s_j \leq
t_j. $

Let $x = z_j,$ for some $j.$ For the integer $s + t^2,$ with $s = s_j $ and $t = t_j$ as above,

$$\frac{1}{s + t^2} \log(|D^uf^{s + t^2}(x)|) = \frac{1}{s + t^2} \log( \prod_{i = 0}^{s-1}| D^uf(f^i(x))| )\cdot  \prod_{i = s}^{s+ t^2 -1} |D^uf(f^i(x))|) $$
$$ \approx \frac {s}{s + t^2} \log(K) + (1 \pm \delta ) \frac{t^2}{s + t^2} (\lambda(\theta_l)) + \frac{r}{s + t^2} \log(K), $$
where $r$ is the rest of the division of $s + t^2 $ by $n $ and $K = \max_{x \in \mathbb{T}^2} {|Df(x)|}. $

So, taking $j = 2n-1 \rightarrow +\infty,$ we get $\frac{1}{s + t^2} \log(|D^uf^{s + t^2}(z_j)|) \approx (1 \pm \delta)\lambda(p),$ analogously taking
$j = 2n \rightarrow +\infty,$ we obtain $\frac{1}{s + t^2} \log(|D^uf^{s + t^2}(z_j)|) \approx (1 \pm \delta)\lambda(q).$

Consider  if $j \geq n,$ and $z_n$ obtained by specification as above. There is an integer $0 < s = s'_j \leq t_j, $ such that for $t = t_j$ holds   $$\frac{1}{s + t^2} \log(|D^uf^{s + t^2}(z_n)|)\approx (1 \pm \delta) \lambda(\theta_j),$$ it is because $z_n$ accompanies $O_j,$ according to specification.

By compactness of $\mathbb{T}$ we can suppose that $z_n \rightarrow z.$ We claim that $z$ is not regular.

By continuity of $D^uf,$ if $j $ is odd, taking $z_n$ enough close to $z,$ with $n \geq j.$ Let $r_j$ be as above for $z_n,$  we obtain $\frac{1}{s + t^2} \log(|D^uf^{r_j} (z)|) \approx (1 \pm \delta)(1 \pm \delta)\lambda(p), $ where $s = s'_j$ and $t = t_j.$ Analogously if  $j $ is even, then  we obtain $\frac{1}{s + t^2} \log(|D^uf^{s + t^2} (z)|) \approx (1 \pm \delta)(1 \pm \delta)\lambda(q), $ since $\delta$ is small we conclude that $z$ is not regular.

\end{proof}

\section{Appendix}
In this appendix, we explore the structure of the proof of Theorem \ref{teo1} to obtain generalization in higher dimensions. For the next result, we define.

\begin{definition}Let $f: \mathbb{T}^d \rightarrow \mathbb{T}^d$ be an Anosov endomorphisms such that there is a $Df-$invariant splitting
$$ T_{x_k}\mathbb{T}^d = \bigoplus_{i = 1}^n E^i_f(x_k: \tilde{x}), k \in \mathbb{Z},$$
for any orbit $\tilde{x} = (x_k)_{k \in \mathbb{Z}}.$ We say that $f$ is $E^i_f-$special if for any orbits $\tilde{x} = (x_k)_{k \in \mathbb{Z}}$ and $\tilde{y} = (y_k)_{k \in \mathbb{Z}},$ such that $x_0 =  y_0,$ we have $E^i_f(x_0: \tilde{x})  =  E^i_f(y_0: \tilde{y}).$
\end{definition}

\begin{maintheorem}\label{teo4} Let $A: \mathbb{T}^d \rightarrow  \mathbb{T}^d, d \geq 3,$ be a linear Anosov endomorphisms, such that $\dim E^s_A \geq 1, \dim E^u_A \geq 1.$ Suppose that $A$ is irreducible over $\mathbb{Q}$ and it has simple real spectrum, such that $E^s_A = E^{s,A}_1 \oplus \ldots \oplus E^{s,A}_{k}$ and $E^u_A = E^{u,A}_{1} \oplus \ldots \oplus E^{u,A}_{n}.$ Consider  $f: \mathbb{T}^d \rightarrow  \mathbb{T}^d$ a smooth Anosov endomorphisms $C^1-$close to $A,$ such that $f$ is  $E^{s,f}_{i}$ and $E^{u,f}_{j}$ special, for $i = 1, \ldots, k$ and $j =  1, \ldots, n,$ the natural continuations of invariant sub bundles of $A.$ Suppose that each  leaf of the foliations $\mathcal{F}^{s,f}_{i}$ and $\mathcal{F}^{u,f}_{j}$ tangent to $E^{s,f}_{i}, i =1, \ldots, k$ and $E^{u,f}_{j}, j =1, \ldots, n$ is non compact. If for any periodic point of $f$ we have coincidence of Lyapunov exponents $\lambda^{s,f}_{i} = \lambda^{s,A}_{i}, i =1, \ldots, k$ and $\lambda^{u,f}_{j} = \lambda^{u,A}_{j}, i =1, \ldots, n,$ then $f$ and $A$ are $C^1-$conjugated.
\end{maintheorem}

\subsection{Proof of Theorem \ref{teo4}}
In the setting of Theorem \ref{teo4}, we can consider the lifts of $\bar{f}$ and $\bar{A},$ it is possible by analogous arguments in Pesin \cite{pesin2004lectures}, we claim that if $f$ is $C^1-$close to $A,$ then at universal cover level $\bar{f}$ admits a similar splitting
$E^s_{\bar{f}} = E^{s,\bar{f}}_1 \oplus E^{s,\bar{f}}_2 \oplus \ldots \oplus E^{s,\bar{f}}_k $ and  $E^u_{\bar{f}} = E^{u,\bar{f}}_1 \oplus E^{u,\bar{f}}_2 \oplus \ldots \oplus E^{u,\bar{f}}_n .$ As before, define $E^{u,\bar{f}}_{(1,i)} = E^{u,\bar{f}}_1 \oplus  \ldots \oplus E^{u,\bar{f}}_i $ and $E^{s,\bar{f}}_{(1,i)} = E^{s,\bar{f}}_1 \oplus  \ldots \oplus E^{s,\bar{f}}_i ,$ analogously, for $j \geq i,$ we define $E^{s,\bar{f}}_{(i, j)}$  and $E^{u,\bar{f}}_{(i, j)}.$

By \cite{pesin2004lectures} of each sub bundle is H\"{o}lder continuous.  We can take the decomposition $E^s_{\bar{f}} \oplus E^{u,\bar{f}}_{(1,i)} \oplus E^{u,\bar{f}}_{(i+1, n)}$ such that it is a uniform partially hyperbolic splitting. Moreover, by \cite{pesin2004lectures}, each $E^{u,\bar{f}}_{(1,i)} = E^{u,\bar{f}}_1 \oplus  \ldots \oplus E^{u,\bar{f}}_i, $ is uniquely integrable to an invariant foliation $W^{u,\bar{f}}_{(1, i)},$ with $i =1, \ldots, n.$ An analogous construction holds for stable directions. Note that $W^{u,\bar{f}}_{(1,i)}(x) \cap W^{u,\bar{f}}(i, n) := W^{u,\bar{f}}_{i}(x)$ tangent to $E^{u,\bar{f}}_i(x).$  The same for stable directions. Define $f-$invariant directions $E^{u,f}_i(x) = D\pi(y)\cdot E^{u,\overline{f}}(y),$ for any $y \in \mathbb{R}^d$ such that $\pi(y) = x.$ The same for stable directions. By hypothesis, if $x - y \in \mathbb{Z}^d$ then $ W^{u,f}_{i}(\pi(x)) = \pi(W^{u,\bar{f}}_{i}(x)) = \pi(W^{u,\bar{f}}_{i}(y)),$ the same for stable directions.  By assumption of Theorem \ref{teo4} each leaf $W^{u,f}_{i}(x), W^{s,f}_{j}(x)$ are non compact leaves.

In the setting of Theorem \ref{teo4}, we can suppose that the eigenvalues of $A$ satisfying $0< |\beta_1^s| < \ldots < |\beta_k^s| < 1 <  |\beta_1^u| < \ldots < |\beta_n^u|. $ The Lyapunov exponents of $A,$ are
$\lambda^s_i(A) = \log(|\beta_i^s|), i =1, \ldots, k$ and $\lambda^u_i(A) = \log(|\beta_i^u|), i =1, \ldots, n.$ For $f$ we denote by $\lambda^u_{i}(x,f)$ the Lyapunov exponent of $f$ at $x$ in the direction $E^{u,f}_i, i = 1, \ldots, n$ and by  $\lambda^s_{i}(x,f)$ the Lyapunov exponent of $f$ at $x$ in the direction $E^{s,f}_i, i = 1, \ldots, k,$ in the cases that Lyapunov exponents are defined.

Let us introduce a notation $E^{s,A}_{(1, i)} = E^{s,A}_1 \oplus \ldots \oplus E^{s,A}_i, i=1, \ldots, k$ and  $E^{u,A}_{(1, i)} = E^{u,A}_1 \oplus \ldots \oplus E^{u,A}_i, i=1, \ldots, n.$ If $j \geq i,$ we denote $E^{s,A}_{(i, j)} =  E^{s,A}_i \oplus \ldots \oplus E^{s,A}_j $ and $E^{u,A}_{(i, j)} =  E^{u,A}_i \oplus \ldots \oplus E^{u,A}_j. $

Let us start with the unstable directions $E^{u,f}_i.$ Fix $i \in \{1, \ldots, n\} $ and  for each $x \in \mathbb{T}^d,$ consider the tangent leaf $W^{u,f}_i(x),$ projected from $\mathbb{R}^d.$ Up to change $f, A$ by $f^2, A^2,$ consider on tangent leaves an orientations such that $f$ and $A$ acts increasingly on $W^{u,f}_i(x)$ and $W^{u,A}_i(x)$ respectively. Since $f$ is $E^{s,f}_i, E^{u,f}_j-$special, it implies that $f$ is strongly special. In fact, by Proposition 2.5 of \cite{MT} we know that $\mathcal{E}^u_f(x),$ the collection of all unstable directions at $x$ is given by $\mathcal{E}^u_f(x) = \overline{\displaystyle\bigcup_{\pi(y) = x} D\pi(y)\cdot(E^u_{\overline{f}}(y))}.$ In the universal cover $\mathbb{R}^d$ holds $E^{u,\overline{f}}(y) = \displaystyle\bigoplus_{j=1}^n E^{u,\overline{f}}_j(y),$ as $f$ is special with respect to each (continuation) bundle $E^{u,f}_j,$ so for $y, y'$ such that $\pi(y) = \pi(y') = x$ then $E^{u,\overline{f}}(y) = E^{u,\overline{f}}(y').$ Applying  Proposition 2.5 of \cite{MT} we conclude that $f$ is special.

Since $A$ is irreducible over $\mathbb{Q}$ the leaves $W^{u,A}_i$ are non-compact and dense on $\mathbb{T}^d.$

To pass from continuity to differentiability we will make an induction process based on the Gogolev method \cite{Go}. In this work, it is proved the following induction steps:

\begin{enumerate}
\item If $h$ is $C^{1 + \nu}$ on $W^{u,f}_{1, m-1}$ and $h(W^{u,f}_i) = W^{u,L}_i, i = 1, \ldots, m-1,$ then $h(W^{u,f}_m) = W^{u,L}_m.$
\item If $h$ is $h(W^{u,f}_m) = W^{u,L}_m, m = 1, \ldots, n,$ then $h$ is $C^{1+ \alpha}$ restricted on each $W^{u,f}_m.$
\end{enumerate}

The proof of the step $(1)$ is topological and the one of step $(2)$ is based on a construction of a Gibbs measure on each leaf $W^{u,f}_m.$ Since the developing in  \cite{GoGu} can be done in the universal cover and the fact that in our case the invariant foliations are invariant by deck transformations, we can assume the topological argument in step $(1).$ In fact, the topological argument in \cite{GoGu}(section 4.4) and \cite{Go}, can be done here with small modifications by using lifts and coherent the inverse branches if it is necessary. Also, since the conjugacy matches invariant one-dimensional foliations, we can prove the step $(2)$ via conformal metrics, as done in the proof of Theorem \ref{teo1}. The following can be proved.

\begin{lemma} Suppose that $h$ is $h(W^{u,L}_m) = W^{u,f}_m, m = 1, \ldots, n,$ then $h$ is $C^{1+ \varepsilon}$ restricted on each $W^{u,f}_m, m = 1, \ldots, n,$ for some $\varepsilon > 0$ enough small.
\end{lemma}

An analogous lemma is true for intermediate stable (uniform contracting) foliations. We conclude the proof of Theorem \ref{teo4} by using Journ\'{e}'s Lemma.



\end{document}